\documentclass[english,12pt,leqno]{amsart}

\usepackage{times,fullpage,amsmath, amstext, amsfonts, amssymb, enumerate, enumitem, hyperref}
\usepackage[utf8]{inputenc}
\usepackage{tikz}

\newtheorem{theorem}{Theorem}
\newtheorem*{maintheorem*}{Theorem \ref{t:main}}
\newtheorem{lemma}[theorem]{Lemma}
\newtheorem*{lemma*}{Lemma}
\newtheorem{proposition}[theorem]{Proposition}
\newtheorem*{proposition*}{Proposition}

\theoremstyle{definition}
\newtheorem{definition}[theorem]{Definition}

\newtheorem{example}[theorem]{Example}
\newtheorem*{example*}{Example}

\theoremstyle{remark}
\newtheorem{remark}[theorem]{Remark}

\newcommand{\eps}{\varepsilon}

\newcommand{\R}{\mathbb{ R}}

\newcommand{\DD}{\mathcal{ D}}
\newcommand{\EE}{\mathcal{ E}}
\newcommand{\WW}{\mathcal{ W}}

\def\lra{{\longrightarrow}}

\DeclareMathOperator{\Sol}{Sol^4_1}
\DeclareMathOperator{\Nil}{Nil}

\newcommand{\N}{\mathbb{N}}
\newcommand{\PP}{\mathbb{P}}

\title{Engel structures and weakly hyperbolic flows on four-manifolds}

\author{D.~Kotschick}
\address{Mathematisches Institut, {\smaller LMU} M\"unchen,
Theresienstr.~39, 80333~M\"unchen, Germany}
\email{dieter@member.ams.org}
\author{T.~Vogel}
\address{Mathematisches Institut, {\smaller LMU} M\"unchen,
Theresienstr.~39, 80333~M\"unchen, Germany}
\email{tvogel@math.lmu.de}


\date{October 12, 2016, revised June 29, 2017; \copyright{D.~Kotschick and T.~Vogel 2016}}

\thanks{We are grateful to Y.~Eliashberg, Y.~Mitsumatsu and E.~Volkov for discussions -- dating back many years -- about this work.
We would also like to thank the anonymous referee for a very careful reading of our paper.}

\begin{document}

\begin{abstract}
We study pairs of Engel structures on four-manifolds whose intersection has constant rank one and which define the same 
even contact structure, but induce different orientations on it. We establish a correspondence between such pairs of Engel structures 
and a class of weakly hyperbolic flows. This correspondence is analogous to the correspondence between bi-contact structures and 
projectively or conformally Anosov flows on three-manifolds found by Eliashberg--Thurston and by Mitsumatsu. 
\end{abstract}

\maketitle

\section{Introduction}

Engel structures are maximally non-integrable two-plane fields $\DD$ on four-manifolds. 
They admit the local normal form $\ker(dz-ydx)\cap\ker(dy-wdx)$ in terms of coordinates $w,x,y,z$. 
Manifolds with Engel structures are parallelisable, and it is known from work of the second author
that all parallelisable four-manifolds do indeed carry Engel structures~\cite{annals}. Moreover, all homotopy 
classes of parallelisations are induced by Engel structures; see R.~Casals, J.~P\'erez, A.~del Pino and F.~Presas~\cite{CP3}.
This makes it interesting to try to understand the geometry of Engel manifolds, and to attempt to single out geometrically
significant ones.

The fact that Engel structures admit a local normal form is one of many properties they share with contact structures.
Another shared property is the stability under sufficiently small perturbations, i.e.~a $C^2$-small perturbation of an Engel structure 
is again an Engel structure. These similarities between contact 
structures and Engel structures suggest that notions from contact topology might have counterparts in the theory of Engel structures. 

In this direction, in this paper we define bi-Engel structures in analogy with the bi-contact structures studied by 
Y.~Eliashberg and W.~Thurston~\cite{ET} and by Y.~Mitsumatsu~\cite{M}. Among other results, these authors showed that 
bi-contact structures correspond to flows satisfying a weak version of hyperbolicity.  We define another notion of weak 
hyperbolicity which allows us to show how to obtain bi-Engel structures from weakly hyperbolic flows and vice versa. 

In Section~\ref{s:Engel} we recall the definitions and simple properties of Engel structures and of even contact structures and 
we introduce bi-Engel structures. Section~\ref{s:hyper} is devoted to flows which are weakly hyperbolic when restricted to a 
smooth invariant subbundle of the tangent bundle. The definition of weak hyperbolicity and the discussion of its most basic 
properties require no assumption on the dimensions of the manifold or the subbundle.

Section~\ref{s:main} contains a detailed proof of our main result:
\begin{theorem} \label{t:main}
Let $\EE$ be an orientable even contact structure on a closed oriented four-manifold $M$, and $\WW$ its characteristic foliation. 
Then $\WW$ is weakly hyperbolic if and only if $\EE$ is induced by a bi-Engel structure $(\DD_+,\DD_-)$.
\end{theorem}
It is clear that with obvious changes of notation our argument also yields the corresponding result whenever 
a one-dimensional foliation $\WW$ is weakly hyperbolic with respect to a rank three subbundle $\EE$, regardless of the dimension
of the ambient manifold. In the case when $\EE$ is the tangent bundle of a three-manifold, one obtains
the correspondence between bi-contact structures and projectively or conformally Anosov
flows discussed in~\cite{ET,M}\footnote{We found the explanations in those references to be somewhat elliptical. Related
arguments also appear in~\cite{CF}.}.

Although bi-contact and bi-Engel structures have very similar definitions and both have relations to flows which are 
weakly hyperbolic in an appropriate sense, there are also important differences. 
As observed first by Mitsumatsu~\cite{M2}, every orientable closed three-manifold has a bi-contact structure. 
More generally, M.~Asaoka, E.~Dufraine and T.~Noda~\cite{ADN} proved that every homotopy class of plane fields 
with trivial Euler class (this is clearly necessary) is realised by bi-contact structures.
For parallelisable four-manifolds we know that Engel structures exist~\cite{annals,CP3}, 
but bi-Engel structures are harder to come by. In contrast to bi-contact structures, the line field of the 
flow associated to a bi-Engel structure is completely determined by one of the two Engel structures, 
in fact by the underlying even contact structure. This makes it difficult to construct examples. 
Nevertheless, in Section~\ref{s:final} we give many examples on mapping tori of contactomorphisms of three-manifolds. 
There are two rather different kinds of examples. The first, which was studied already in~\cite{thurston}, and which was 
one of the motivations for this paper, is the Thurston geometry $\Sol$, including mapping tori of $\Nil^3$-manifolds. 
The second consists of suspensions of contact Anosov flows, which are plentiful
according to the work of P.~Foulon and B.~Hasselblatt~\cite{FH}.

An outstanding problem about Engel structures, again in parallel with three-dimensional contact topology, is whether 
there is a useful notion of tightness for them. While we do not directly address this question here, we will 
in Subsection~\ref{ss:rigid} discuss a remarkable rigidity property of the flow lines of the characteristic foliation of 
certain Engel structures, which follows from work of R.~Bryant and L.~Hsu~\cite{BH}; compare the 
very recent~\cite{PP}. Remarkably, this rigidity property is tautologically satisfied for bi-Engel structures, which may or 
may not provide a useful hint towards isolating non-flexible properties which may distinguish between different
kinds of Engel structures.

\section{Engel and bi-Engel structures}\label{s:Engel}

This section contains the definitions and elementary facts about the distributions appearing in this note. More information about even contact structures can be found for example in~\cite{McD}, while \cite{golubev,montgomery} and \cite{annals} contain background on Engel structures. 

\subsection{Even contact structures}\label{ss:even}

\begin{definition}\label{d:even}
An even contact structure on a $2n$-dimensional manifold $M$ is a maximally non-integrable smooth hyperplane field $\EE$. 
\end{definition}

Such a hyperplane field can be defined locally by a one-form $\alpha$ with the property that $\alpha\wedge (d\alpha)^{n-1}$ is nowhere zero. A global defining form exists if and only if $\EE$ is coorientable. The two-form $d\alpha$ has maximal rank on $\EE$. If one changes the defining form $\alpha$, then the restriction of $d\alpha$ to $\EE$ changes only by multiplication with a function, so its conformal class is intrinsically defined. The kernel of $d\alpha$ restricted to $\EE$  coincides with the kernel of the $(2n-1)$-form $\alpha\wedge (d\alpha)^{n-1}$. This kernel is a line field $\WW\subset\EE$ giving rise to the characteristic foliation of $\EE$, and the quotient bundle $\EE/\WW$ carries a conformal symplectic structure. The form $(d\alpha)^{n-1}$ gives
$\EE/\WW$ an orientation independent of choices precisely when $n$ is odd.

If $W$ is any vector field tangent to $\WW$, then
$$
L_W\alpha = di_W\alpha+i_Wd\alpha = i_Wd\alpha
$$
vanishes on $\EE$, and is therefore a multiple of $\alpha$. Thus any flow tangent to the characteristic foliation $\WW$ preserves $\EE=\ker\alpha$. 

\begin{lemma}\label{l:evenorient}
If $n$ is even, the orientability of $M$ is equivalent to the orientability of $\WW$.
\end{lemma}

\begin{proof}
Note that $\EE/\WW$ defines a contact structure on transversals to $\WW$, and therefore orients the transversals canonically exactly when $n$ is even. The holonomy of $\WW$ preserves this orientation. Therefore $\WW$ is orientable if and only if $TM$ is. 
\end{proof}

We now discuss the condition for the existence of a defining form $\alpha$ for $\EE$ which is preserved by the holonomy of the characteristic foliation.

\begin{lemma}\label{l:volume}
Let $\EE$ be a coorientable even contact structure, with characteristic foliation $\WW$. The following conditions are 
equivalent:
\begin{enumerate}
\item The defining form $\alpha$ for $\EE$ can be  chosen such that $d\alpha$ is of constant rank $2n-2$.
\item The characteristic foliation $\WW$ is the kernel of a closed $(2n-1)$-form.
\item The characteristic foliation $\WW$ has volume-preserving holonomy.
\end{enumerate}
\end{lemma}

\begin{proof}
The equivalence of the second and third conditions is well known; both conditions amount to saying that a spanning vector field
is divergence-free with respect to a suitable volume form. 

We prove the equivalence of the first two conditions.
If $d\alpha$ is of constant rank $2n-2$, then $\alpha\wedge (d\alpha)^{n-1}$ is a closed $(2n-1)$-form with kernel $\WW$.
Conversely, suppose that $\beta$ is an arbitrary defining form for $\EE$, and that $\gamma$ is a closed $(2n-1)$-form 
with kernel $\WW$. Then $\beta\wedge (d\beta)^{n-1}$ is another $(2n-1)$-form with kernel $\WW$, and after replacing $\gamma$ 
by its negative if necessary, we  see that 
$$
\gamma = f \beta\wedge (d\beta)^{n-1}
$$
for some positive smooth function $f$ on $M$. Set $\alpha = f^{1/n}\beta$. This is a defining form for $\EE$, with $(d\alpha)^{n}$ identically zero. The rank of $d\alpha$ is therefore strictly smaller than $2n$, and as it can not be smaller than $2n-2$, it is $2n-2$ everywhere.
\end{proof}

In the situation of this lemma, if $\alpha$ is chosen such that $d\alpha$ is of rank $2n-2$, and $W$ is tangent to $\WW$, then $L_W\alpha=i_Wd\alpha$ vanishes, as $\WW$ is in the kernel of $d\alpha$. Thus the flow of $W$ preserves the form $\alpha$, and not just its kernel.

\subsection{Engel structures}\label{ss:Engel}

\begin{definition}\label{d:Engel}
An Engel structure on a $4$-dimensional manifold $M$ is a smooth rank $2$ distribution $\DD$ with the property that $[\DD,\DD]$ is an even contact structure $\EE$. 
\end{definition}

If $\EE$ is an even contact structure and $\DD$ is an Engel structure whose derived distribution $[\DD,\DD]$ coincides with $\EE$, we say that $\EE$ is induced by $\DD$, and that $\DD$ is subordinate to $\EE$. 

\begin{lemma}\label{l:WinD}
If $\DD$ is subordinate to $\EE$, then the characteristic foliation $\WW$ of $\EE$ is contained in $\DD$. 
\end{lemma}

\begin{proof}
We argue by contradiction. If $p\in M$ is a point  with $\WW_p$ not contained in $\DD_p$, we choose a local frame $X$, $Y$ for $\DD$ around $p$, and a local defining form $\alpha$ for $\EE$. Then $d\alpha$ is non-degenerate on
$\DD_p$, and so $d\alpha(X,Y)$ does not vanish at $p$. Therefore
$$
\alpha([X,Y])=L_X(\alpha(Y))-L_Y(\alpha(X))-d\alpha(X,Y)=-d\alpha(X,Y)\neq 0 \ ,
$$
contradicting $[X,Y]\in\EE=\ker\alpha$.
\end{proof}

We now discuss orientations for the distributions involved in the definition of an Engel structure subordinate to a given
even contact structure.


\begin{lemma}\label{l:orient}
{\rm 1.} Every Engel structure defines a canonical orientation on its induced even contact structure.

{\rm 2.} The following conditions on a $4$-manifold $M$ endowed with an Engel structure are equivalent:
\begin{itemize}
\item[(a)] $M$ is orientable,
\item[(b)] $\WW$ is orientable,
\item[(c)] $\EE$ is coorientable.
\end{itemize}
\end{lemma}

\begin{proof}
Suppose that $X$ and $Y$ are vector fields forming a local frame for an Engel structure $\DD$. Then $X$, $Y$ and $[X,Y]$ form a local frame for the induced even contact structure, and the local orientation of $\EE$ given by this frame is independent of the choice of $X$ and $Y$. This proves the first statement.

The equivalence of (a) and (c) follows immediately from what we just proved. The equivalence of (a) and (b) was proved in Lemma~\ref{l:evenorient}.
\end{proof}

\subsection{Bi-Engel structures}\label{ss:biEngel}

The first part of Lemma~\ref{l:orient} motivates the following:

\begin{definition}\label{d:biEngel}
A bi-Engel structure on a $4$-dimensional manifold $M$ is a pair of Engel structures  $(\DD_+,\DD_-)$ inducing the same even contact structure $\EE$, defining opposite orientations for $\EE$, and having one-dimensional intersection. 
\end{definition}

By Lemma~\ref{l:WinD}, the two Engel structures making up a bi-Engel structure must both contain the characteristic foliation $\WW$ of the induced even contact structure $\EE$. Thus their intersection is precisely $\WW$, and their span is $\EE$. 

The geometric meaning of the definitions of Engel and bi-Engel structures can be elucidated as follows. The holonomy of the characteristic foliation $\WW$ of an even contact structure $\EE$ preserves $\EE$. An Engel structure $\DD$ subordinate to $\EE$ is a plane field inside $\EE$, which turns in a fixed direction around the axis $\WW$ under the holonomy of $\WW$. Specifying the direction in which $\DD$ turns amounts to specifying an orientation for $\EE$. The two Engel structures $\DD_{\pm}$ making up a bi-Engel structure intersect in $\WW$, and rotate around it in opposite directions under the holonomy of $\WW$. Moreover, the condition that the two Engel planes never coincide, prevents them from making full turns around $\WW$. This means that for the flow $\varphi_t$ of a spanning vector field for $\WW$ one has $D\varphi_{-t}(\DD (\varphi_t(p)))\neq \DD (p)$ for all $t\neq 0$.

To end this section, we point out that the requirement that $\DD_+\cap\DD_-$ be one-dimensional can not be omitted from Definition~\ref{d:biEngel}. 
If two Engel structures, not necessarily forming a bi-Engel structure, are subordinate to the same even contact structure $\EE$ and define opposite orientations of $\EE$, then they turn in opposite directions under the holonomy of the characteristic foliation. Therefore, on every leaf of $\WW$ the points where the two Engel distributions coincide form a discrete subset of the leaf. In particular, the two Engel distributions are different almost everywhere, but it is possible that they coincide at some points. This is what happens in the following example,
which is a variation on the classical prolongation, cf.~\cite{montgomery}.

\begin{example}
Let $N$ be a closed $3$-manifold and $\xi$ a contact structure which is trivial as a vector bundle over $N$. Pick a global framing of $\xi$ by vector fields $X$ and $Y$. Consider $S^1$ with coordinate $t\in\R$ modulo $2\pi$, and let $M=N\times S^1$. The distribution $\EE=\xi\oplus TS^1$ is an even contact structure on $M$ with characteristic foliation $\WW=TS^1=\R\frac{\partial}{\partial t}$.

Let $\DD_{\pm}$ be the span of $\WW$ and $cos(t)\cdot X\pm sin(t)\cdot Y$. Then the $\DD_{\pm}$ are Engel structures subordinate to $\EE$, but inducing opposite orientations on $\EE$. However, they do not form a bi-Engel structure because they agree at the points where $sin(t)=0$.
\end{example} 

\section{Weakly hyperbolic flows}\label{s:hyper}

In this section we introduce a weak notion of hyperbolicity for flows which are tangent to a fixed distribution, and which preserve this distribution.

Let $M$ be a closed manifold, $\EE\subset TM$ a smooth subbundle, and $\WW\subset\EE$ an orientable line field with $[\WW,\EE]\subset\EE$. This ensures that $\EE$ is preserved by any flow tangent to $\WW$. Moreover, such a flow then acts on the quotient bundle $\EE/\WW$.

\begin{definition}\label{d:hyper}
The flow $\varphi_t$ on $M$ generated by a non-zero vector field $W$ spanning $\WW$ is said to be weakly hyperbolic if there are constants $K,c>0$ and a continuous metric on $\EE/\WW$ such that for all $p\in M$ there is a decomposition
$$
\EE(p)/\WW(p) = \EE_+(p) \oplus \EE_-(p)
$$
for which the following inequality holds for all $t>0$ and $0\neq v_\pm\in \EE_\pm$  
\begin{equation}\label{eq:hyper}
\frac{\vert\vert D\varphi_t(v_+)\vert\vert}{\vert\vert v_+\vert\vert} \geq Ke^{ct} \ \frac{\vert\vert D\varphi_t(v_-)\vert\vert}{\vert\vert v_-\vert\vert} \ .
\end{equation}
\end{definition}

This condition is independent of the spanning vector field $W$ chosen for $\WW$, as long as we fix an orientation for $\WW$. It is also independent of 
the choice of metric $g$, cf.~\cite{AnS}.
\begin{remark} \label{r:mit und ohne K}
If $\varphi$ is weakly hyperbolic with respect to the metric $g$, then after replacing $g$ by $1/T\int_0^T\varphi_t^*g\,dt$ one can choose $K=1$ if $T$ is large enough.
\end{remark}

\begin{lemma} \label{l:Epm are bundles}
The subspaces $\EE_\pm(p)$ for $p\in M$ of $\EE/\WW$ in Definition~\ref{d:hyper} are $\varphi_t$--invariant, have constant dimension and 
depend continuously on $p$. 
\end{lemma}
\begin{proof} 
The proof is a modification of a proof in \cite{AnS}, p.~121. 

Let first $p\in M$ be arbitrary.
Note that if $0\neq X\in\EE_-(p)$, then for all $Y\in\EE\setminus\EE_-$ there are constants $T_Y,K_Y>0$ depending only on the angle between 
$Y$ and $\EE_-$ (and $K,c$, of course) such that
\begin{equation} \label{eq:hyper2 mod}
\frac{\|D\varphi_t(Y)\|}{\|D\varphi_t(X)\|}\ge K_Ye^{ct}\frac{\|Y\|}{\|X\|}
\end{equation}
for $t>T_Y$.

For the verification let $X\in\EE_-$ and fix $Y\in\EE\setminus\EE_-$. We write $Y=Y_++Y_-$ with $Y_\pm\in\EE_\pm, 0\neq Y_+$ and $\kappa>0$ such that $\|Y_-\|\le\kappa\|Y_+\|$. By \eqref{eq:hyper} we have 
$$
\frac{\|D\varphi_t(Y_-)\|}{\|D\varphi_t(Y_+)\|}\le K^{-1}e^{-ct}\kappa<1
$$
where the last inequality holds for large enough $t$.  Then because of
\begin{align*}
\|Y\| & \le \|Y_+\|+\|Y_-\| \le (1+\kappa)\|Y_+\|
\end{align*} 
and \eqref{eq:hyper} we get
\begin{align*}
\frac{\|D\varphi_t(Y)\|}{\|D\varphi_t(X)\|}  &  \ge \frac{\|D\varphi_t(Y_+)\|-K^{-1}e^{-ct}\kappa \|D\varphi_t(Y_+)\|}{\|D\varphi_t(X)\|}  \\
& \ge \frac{K-\kappa e^{-ct}}{1+\kappa}e^{ct}\frac{\|Y\|}{\|X\|} \ .
\end{align*}
Thus we can  choose $T_Y$ so large that $K>2\kappa e^{-cT_Y}$ and $K_Y=\frac{K}{2(1+\kappa)}$. These constants depend only on $K,c$ and $\kappa$. 

Now we show that $\EE_-$ is continuous at $p\in M$. Let $p_n$ be a sequence converging to $p$. 
After passing to a subsequence we may assume that 
$\lim_{n\to\infty}[\EE_\pm(p_n)]=[\EE'_\pm]$ for some $\EE'_\pm$, and that $\dim(\EE_+(p_n))$ is constant. 
Since $\dim(\EE_+(p_n))$ and $\dim(\EE_-(p_n))$
have constant sum ($=\dim(M)-1$), the latter is also constant.

Let us assume that $\EE'_-$ is not contained in $\EE_-(p)$. Then we may fix sequences $X_n\in\EE_-(p_n), Y_n\in\EE(p_n)$ such that $\lim_{n\to\infty}X_n=X\notin\EE_-(p)$ and $\lim_{n\to\infty}Y_n=Y\in\EE_-(p)$. In particular, we may assume that the angle between $Y_n$ and $\EE_-(p_n)$ is uniformly bounded away from $0$. This means that for $Y_n=Y_{n+}+Y_{n-}, Y_{n\pm}\in\EE_\pm(p_n)$ the ratio $\|Y_{n-}\|/\|Y_{n+}\|$ is bounded from above by a constant $\kappa>0$ which is independent of $n$. The constants $T_Y, K_Y$ appearing in~\eqref{eq:hyper2 mod} actually depend only on $c,K$ and $\kappa$, thus choosing $T_Y$ independently of $n$ such that for $t>T_Y$ we have
\begin{align*}
\frac{\|D\varphi_t(Y_n)\|}{\|D\varphi_t(X_n)\|} & \ge K_Ye^{ct} \frac{\|Y_n\|}{\|X_n\|} \\
\frac{\|D\varphi_t(X)\|}{\|D\varphi_t(Y)\|} & \ge K_Ye^{ct} \frac{\|X\|}{\|Y\|}.
\end{align*}
Since $\varphi_t$ is smooth, we get a contradiction if $t$ satisfies $K_Ye^{ct}>1$ as $n$ goes to $\infty$. This implies $\EE'_-\subset\EE_-(p)$. 

Considering $\varphi_{-t}$ instead of $\varphi_t$ one shows $\EE_+\subset\EE_+(p)$. The fact that 
$\dim(\EE_+(p))+\dim(\EE_-(p))=\dim(M)-1=\dim(\EE_+')+\dim(\EE_-')$ then implies $\EE_\pm'=\EE_\pm(p)$. 

The $\varphi_t$--invariance of the bundles $\EE_-$ now follows from the property described in \eqref{eq:hyper2 mod} since this property can be used to characterize the elements of $\EE_-$.   
\end{proof}

If we change the orientation of $\WW$, by replacing $W$ with $-W$, say, then weak hyperbolicity is preserved, but the roles of $\EE_{\pm}$ are interchanged. The holonomy of $\WW$ preserves $\EE$ and acts naturally on the quotient $\EE/\WW$, and the condition in the definition is that the holonomy is much more expanding on $\EE_+$ than on $\EE_-$. 
This does not preclude the possibility that the holonomy could be expanding (or contracting) on both $\EE_{\pm}$, as long as the expansion (or contraction) rates are such that~\eqref{eq:hyper} is satisfied. In the case that $\EE$ is the tangent bundle of a three-manifold, Definition~\ref{d:hyper} reduces to the definition of flows that are conformally Anosov~\cite{ET} or projectively Anosov (pA)~\cite{M}.

By an obvious simplification of terminology, we call $\WW$ weakly hyperbolic, without saying something like ``weakly hyperbolic with  respect to $\EE$''. A given line field $\WW$ may of course preserve several distributions it is contained in, and be weakly hyperbolic for some but not for others. However, it will always be clear which distribution is used for $\EE$ when discussing weak hyperbolicity of $\WW$.

If the distribution $\EE$ is integrable, then it defines a foliation, and a flow tangent to $\WW\subset\EE$ restricts to every leaf of this foliation. The flow is weakly hyperbolic in the sense of Definition~\ref{d:hyper} if and only if its restriction to every leaf is conformally Anosov. 

For the purposes of this paper we are interested in the case when $\EE$ is an even contact structure, and $\WW$ is its 
characteristic foliation. If the dimension of $M$ is four, then $\EE$ has rank three, and the subbundles $\EE_{\pm}$ are 
actually line fields. However, even in higher dimensions, when these subbundles have higher rank, they tend to have a 
very specific geometry. We shall return to this in Subsection~\ref{ss:contactAnosov} below.


\section{Proof of the main theorem}\label{s:main}

In this section we prove Theorem~\ref{t:main}. 
In the proof we shall use some facts about the cross ratio. 
One of the numerous sources for this material is~\cite{Be}.

Let $V$ be a real vector space of dimension $2$. If $x_1,x_2,x_3\in\mathbb{P}(V)$ are distinct and $z\in\mathbb{P}(V)$ is arbitrary, then the cross ratio $[x_1,x_2,x_3,z]\in\R\mathbb{P}^1$ is the image of $z$ under the unique homography $f: \mathbb{P}(V)\lra\R\mathbb{P}^1$ with $f(x_1)=[1:0], f(x_2)=[0:1] ,f(x_3)=[1:1]$. In particular, if $f : V \lra V'$ is a linear isomorphism (in our application of the cross ratio $f$ will be the linearized holonomy of a foliation of rank $1$) and $\underline{f}$ is the induced map between projective spaces, then 
$$
[x_1,x_2,x_3,z]=\left[\underline{f}(x_1),\underline{f}(x_2),\underline{f}(x_3),\underline{f}(z)\right] \ .
$$
After identifying $\R\mathbb{P}^1\setminus[1:0]$ with the real numbers, we can treat the cross ratio as a number unless $z=x_1$. 
In other words, $[x_1,x_2,x_3,z]=[1:0]\hat{=}\infty$ if and only if $z=x_1$. 

If $x_1,x_2,x_3,z\in\mathbb{P}(V)\setminus\{\mathrm{pt}\}$ are pairwise distinct, then the cross ratio $[x_1,x_2,x_3,z]\in\R\mathbb{P}^1\setminus\{\infty\}$ can be computed in terms of affine coordinates on $\mathbb{P}(V)\setminus\{\mathrm{pt}\}$ as follows:
$$
[x_1,x_2,x_3,z]=\frac{(x_3-x_1)(z-x_2)}{(z-x_1)(x_3-x_2)} \ .
$$
Using this formula one can show the following relation for pairwise distinct points $x,a,a',b',b,y$ of $\PP(V)$
$$
[x,a',b',y]=[x,a,b,y] \cdot [a,a',b',b] \cdot [a,a',b,y] \cdot [x,a,b',b] \ .
$$
In particular, when the points $x,a,a',b',b,y$ lie in this order on the projective line $\mathbb{P}(V)$, then it follows from the definition of the cross ratio that 
\begin{align*}
[a,a',b,y]& >1 & [x,a,b',b] & >1 \ .
\end{align*}
Therefore we obtain the following inequality if the assumption on the ordering of $x,a,a',b',b,y$ is satisfied:
\begin{equation} \label{doppelvrel}
[x,a',b',y] > [x,a,b,y] \cdot [a,a',b',b] \ .
\end{equation} 

We can finally prove our main result.

\begin{proof}[Proof of Theorem~\ref{t:main}]
Recall that by Lemma~\ref{l:evenorient} the characteristic foliation $\WW$ is orientable if and only if the same is true for $M$. 

Let $\EE$ be an orientable even contact structure whose characteristic foliation $\WW$ is weakly hyperbolic and oriented. 
We fix a positive spanning vector field $W$ for $\WW$ and denote its flow by $\varphi_t$. We also fix the splitting 
$\EE/\WW = \EE_+ \oplus \EE_-$, a continuous metric $g$ and 
constants $c$ and $K$ as in the definition of weak hyperbolicity. By Remark~\ref{r:mit und ohne K} we may assume $K=1$. 

Assume first that the line fields $\EE_{\pm}$ are orientable, and that $X_{\pm}$ are sections of $\EE$ projecting to $\EE/\WW$ 
as spanning vector fields for $\EE_{\pm}$, of unit length with respect to $g$, say. As the line fields $\EE_{\pm}$ are invariant 
under the flow of $W$, we find that there are continuous real-valued functions $\lambda_{\pm}(t,p)$ on $\R\times M$ such that
$$
D_p\varphi_t (X_{\pm}(p)) = \lambda_{\pm} (t,p)X_{\pm}(\varphi(t)) \mod \WW.
$$
That $\varphi_t$ is a flow implies $\lambda_{\pm} (0,p)=1$, and 
$$
\lambda_{\pm}(t,\varphi_s(p)) \cdot \lambda_{\pm}(s,p) = \lambda_{\pm}(t+s,p)
$$
for all $p\in M$. The definition of weak hyperbolicity of the flow in this case means that there is a constant $c>0$ such that  
\begin{equation}\label{eq:wh}
\lambda_+(t,p)\geq e^{ct}\lambda_-(t,p)
\end{equation}
for all $p\in M$ and all $t\geq 0$.

If we assume that the vector fields $X_{\pm}$ are smooth, then so are the functions $\lambda_{\pm}$. In this case, by differentiating 
at $0\in\R$, the inequality~\eqref{eq:wh} implies 
\begin{equation}\label{eq:whdiff}
\lambda '_+(0,p)\geq c+\lambda '_-(0,p).
\end{equation}
We can define smooth rank two subbundles $\DD_{\pm}\subset\EE$ as the span of $W$ and $X_+\pm X_-$. Using the smoothness assumption, we can calculate commutators:
\begin{alignat*}{1}
[W,X_{\pm}](p) &= \left.\frac{d}{dt}\right|_{t=0}(D\varphi_{-t})\left(X_{\pm}(\varphi_t(p))\right) \\
&= \left.\frac{d}{dt}\right|_{t=0}\left( \frac{1}{\lambda_{\pm}(t,p)}X_{\pm}(p)\right) \mod \WW \\
&= -\lambda'_{\pm}(0,p)X_{\pm}(p).
\end{alignat*}
It follows that 
$$
[W,X_+\pm X_-](p)+\lambda '_+(0,p)(X_+\pm X_-)(p) = \pm (\lambda '_+(0,p)-\lambda '_-(0,p))X_-(p) \mod \WW.
$$
Combining this with~\eqref{eq:whdiff} we see that the $\DD_{\pm}$ are Engel structures subordinate to $\EE$ and that they induce opposite orientations of $\EE$. Thus they form a bi-Engel structure.

Now let us consider the case when the $X_\pm$ are only continuous, not necessarily smooth. 
In this case we first show that we may assume the $X_\pm$ 
to have continuous first and second derivatives along the flow lines of $W$. 
To achieve this we fix a mollifier, i.e.~a smooth function $h \colon\R\lra\R^+_0$ with support in 
$[-1,1]$ and $\int_{\R} h(s)ds=1$, and consider the usual convolution
\begin{equation} \label{eq:glatt entlang fluss}
\big(h*X_\pm\big)(p)=\int_{\R} h(s)D\varphi_s\left(X_\pm\left(\varphi_{-s}(p)\right)\right)ds \ .
\end{equation}
By definition $h*X_\pm$ is a section of $\EE_\pm$ which is nowhere tangent to $\mathcal{W}$. 
When $h(s)$ is replaced by $h_\kappa(s)=\kappa h(\kappa s)$ in~\eqref{eq:glatt entlang fluss} then $h_\kappa*X_\pm$ converges uniformly to 
$X_\pm$ as $\kappa\to\infty$.  
Moreover, the restrictions of $h*X_\pm$ to segments of $\WW$ are smooth when viewed as sections of the smooth bundle 
$\EE/\WW$. The derivatives
$$
L_W(h*X_\pm)(p) =\lim_{\eta\to 0}\frac{D\varphi_{-\eta}\left(h*X_{\pm}\right)(\varphi_{\eta}(p)) -\left(h*X_{\pm}\right)(p) }{\eta} 
$$
are continuous on $M$ (not only along the leaves of $\WW$), the same is true for derivatives of higher order.   

We choose smooth sections $Z_\pm$ of $\EE$ which are $C^0$ close to $X_\pm$ and such that the first and second derivatives along 
$\WW$ are also close to those of $X_\pm$. There are continuous functions $w_\pm,s_\pm,u_\pm$ which are $C^2$ along the leaves of 
$\WW$ such that 
$$
Z_\pm=w_\pm W+s_\pm X_-+u_\pm X_+ \ .
$$
Because $Z_+$ approximates $X_+$, the function $s_+$ is $C^0$-close to $0$ and $u_+$ is $C^0$-close to $1$, and similarly 
for the approximation of $X_-$ by $Z_-$.
Their first derivatives in the direction of $\WW$ are close to zero. Therefore the calculation of commutators performed with $Z_\pm$ in place of $X_\pm$ shows that $W$ and $Z_+\pm Z_-$ span two Engel structures $\DD_\pm$ subordinate to $\EE$ and inducing opposite orientations on $\EE$. 

Finally if the line bundles $\EE_\pm$ are non-trivial, then we can only choose $X_\pm$ up to sign. Nevertheless, the functions $\lambda_\pm$ are well-defined, and the whole argument goes through by using the approximating sections $Z_\pm$ to be invariant under sign change. Thus we have proved that an even contact structure with weakly hyperbolic characteristic foliation has a subordinate bi-Engel structure.

It remains to prove the converse. Let $(\DD_+,\DD_-)$ be a bi-Engel structure subordinate to $\EE$ and $W$ a vector field spanning the characteristic foliation $\WW\subset\EE$. 
Then the flow $\varphi_t$ of $W$ preserves $\EE$. In order to show that this is weakly hyperbolic we have to find a splitting 
$\EE/\WW=\EE_+\oplus\EE_-$ such that~\eqref{eq:hyper} holds. This is done in two steps. First we find invariant plane fields 
$\DD^\infty$ and $\DD^{-\infty}$ whose intersection is $\WW$. Then we check weak hyperbolicity for the induced splitting with 
$\EE_{\pm}=\DD^{\pm\infty}/\WW$. 

For the plane fields $\DD^{\pm\infty}$ we have candidates
\begin{align*} 
\DD_\pm^{\infty}(p) & =\lim_{t\to\infty}D\varphi_{-t}\left(\DD_\pm(\varphi_t(p))\right)  
\\
\DD_\pm^{-\infty}(p)& =\lim_{t\to-\infty}D\varphi_{-t}\left(\DD_\pm(\varphi_t(p)\right).
\end{align*}
for $p\in M$. Each of these limits exists. We explain this for $\DD_+^\infty$. Let $p\in M$ and consider the planes $D\varphi_{-t}\left(\DD_{+}(\varphi_{t}(p))\right)$ and  $D\varphi_{-t}\left(\DD_{-}(\varphi_{t}(p))\right)$ in $\EE(p)$. 
Both of them contain $W(p)$ and the fact that $\DD_{\pm}$ are Engel structures inducing opposite orientations of $\EE$ ensures that these planes rotate without stopping around $\WW$ in opposite directions as $t$ increases. Since they are always transverse to each other this implies that the limit defining $\DD_+^\infty$ exists. 

\begin{center}
\begin{figure}[htb] 
\begin{tikzpicture}[y=0.80pt, x=0.8pt,yscale=-1, inner sep=0pt, outer sep=0pt]
  \path[draw=black,line join=miter,line cap=butt,line width=0.658pt]
    (347.6554,78.6456) -- (347.6554,282.3446);
  \path[draw=black,line join=miter,line cap=butt,line width=0.800pt]
    (157.7959,195.2811) -- (560.5908,195.2811);
  \path[draw=black,line join=miter,line cap=butt,miter limit=4.00,line
    width=0.715pt] (530.7460,124.6348) -- (209.5463,249.5849);
  \path[draw=black,dash pattern=on 2.18pt off 2.18pt,line join=miter,line
    cap=butt,miter limit=4.00,line width=0.727pt] (475.7298,90.6473) --
    (246.6346,277.1282);
  \path[draw=black,dash pattern=on 2.16pt off 2.16pt,line join=miter,line
    cap=butt,miter limit=4.00,line width=0.719pt] (284.1811,86.2759) --
    (391.9036,275.5104);
  \path[draw=black,line join=miter,line cap=butt,line width=0.800pt]
    (217.8719,132.7198) -- (488.0095,263.6778);
  \path[fill=black] (502.74252,181.46201) node[above right] (text5323)
    {$\mathcal{E}_+(p)=\mathcal{D}^{-\infty}$};
  \path[fill=black] (543.21155,122.14918) node[above right] (text5327)
    {$\mathcal{D}_+(p)$};
  \path[fill=black] (485.59116,97.670372) node[above right] (text5331)
    {$D\varphi_{-t}(\mathcal{D}_+(\varphi_t(p)))$};
  \path[fill=black] (351.25247,105.15765) node[above right] (text5335)
    {$\mathcal{E}_-(p)=\mathcal{D}^\infty$};
  \path[fill=black] (224.33167,124.4982) node[above right] (text5339) {};
  \path[fill=black] (205.53947,123.3237) node[above right] (text5343)
    {$\mathcal{D}_-(p)$};
  \path[fill=black] (222.69386,79.713348) node[above right] (text5347)
    {$D\varphi_{-t}(\mathcal{D}_-(\varphi_t(p)))$};
  \path[draw=black,line join=miter,line cap=butt,line width=0.800pt]
    (555.2445,192.1926) -- (555.2445,192.1926);
  \begin{scope}[cm={{1.00842,0.0,0.0,1.0,(-4.63463,0.0)}}]
    \path[draw=black,line join=miter,line cap=butt,line width=0.800pt]
      (560.2275,195.3330) -- (550.2614,189.0523);
    \path[draw=black,line join=miter,line cap=butt,line width=0.800pt]
      (560.2794,195.2551) -- (550.3133,201.5358);
  \end{scope}
  \begin{scope}[cm={{0.0,-1.0,1.0,0.0,(152.15966,639.16537)}}]
    \path[draw=black,line join=miter,line cap=butt,line width=0.800pt]
      (560.2275,195.3330) -- (550.2614,189.0523);
    \path[draw=black,line join=miter,line cap=butt,line width=0.800pt]
      (560.2794,195.2551) -- (550.3133,201.5358);
  \end{scope}
  \path[shift={(0.45594,-1.36781)},draw=black,dash pattern=on 0.80pt off
    1.60pt,miter limit=4.00,line width=0.800pt]
    (284.4858,180.4035)arc(193.912:255.174:61.551);
  \path[cm={{-1.0,0.0,0.0,1.0,(714.37766,-8.09306)}},draw=black,dash pattern=on
    0.80pt off 1.60pt,miter limit=4.00,draw opacity=0.984,line width=0.800pt]
    (285.6129,186.0911)arc(188.513:240.032:59.271652 and 61.551);
  \begin{scope}[cm={{0.94195,-0.33577,0.33577,0.94195,(-264.29004,138.29416)}}]
    \path[draw=black,dash pattern=on 0.80pt off 0.80pt,line join=miter,line
      cap=butt,miter limit=4.00,line width=0.800pt] (560.2275,195.3330) --
      (550.2614,189.0523);
    \path[draw=black,dash pattern=on 0.80pt off 0.80pt,line join=miter,line
      cap=butt,miter limit=4.00,line width=0.800pt] (560.2794,195.2551) --
      (550.3133,201.5358);
  \end{scope}
  \begin{scope}[cm={{-0.78444,-0.62021,0.62021,-0.78444,(716.67326,633.2144)}}]
    \path[draw=black,dash pattern=on 0.80pt off 0.80pt,line join=miter,line
      cap=butt,miter limit=4.00,line width=0.800pt] (560.2275,195.3330) --
      (550.2614,189.0523);
    \path[draw=black,dash pattern=on 0.80pt off 0.80pt,line join=miter,line
      cap=butt,miter limit=4.00,line width=0.800pt] (560.2794,195.2551) --
      (550.3133,201.5358);
  \end{scope}

\end{tikzpicture}

\caption{Configuration of lines in $\EE(p)/\WW(p)$\label{b:lines}}
\end{figure}
\end{center}

Let us now show that $\DD_+^{\infty}(p)=\DD_-^{\infty}(p)$ for all $p\in M$. Since $M$ is compact, there exists a sequence $(t(i))_{i\in\N}$ and $q\in M$ such that $\lim_{i\to\infty}t(i)=\infty$ and $\lim_{i\to\infty}\varphi_{t(i)}(p)=q$. Fix a compact local transversal $C$ of $\WW$ through $q$ and $\eps>0$ such that
\begin{alignat*}{2}
C \times &[-\eps,\eps] &&\longrightarrow M \\
(c\ , &\ \tau) &&\longmapsto \varphi_\tau(c)
\end{alignat*}
is an embedding. For $t\in\R$ let 
$$
d_\pm(t)=\left[D\varphi_{-t}\left(\DD_\pm(\varphi_{t}(p))\right)\right]\in\mathbb{P}(\EE(p)/\WW(p)).
$$
Recall that $\DD_+\cap\DD_-=\WW$. Because $\DD_+,\DD_-$ are Engel structures which induce opposite orientations of $\EE$, it follows that for $0<t<s$, the lines $d_+(0),d_+(t),d_+(s),d_-(s),d_-(t),d_-(0)$ are ordered in this way on $\PP(\EE(p)/\WW(p))$ and these six lines are all distinct. 
In particular, all cross ratios below take values in $(1,\infty )$.
By compactness of $C$ there is an $\alpha>1$ such that 
$$
[d_+(t(i)-\eps),d_+(t(i)+\eps),d_-(t(i)+\eps),d_-(t(i)-\eps)]>\alpha
$$
for all $i$. According to \eqref{doppelvrel}  
\begin{align*}
[d_+(0),d_+(t(i)+\eps),d_-(t(i)+\eps),d_-(0)] & > \alpha[d_+(0),d_+(t(i)-\eps),d_-(t(i)-\eps),d_-(0)] \\ 
& >\alpha[d_+(0),d_+(t(i-1)+\eps),d_-(t(i-1)+\eps),d_-(0)]\\
& >\ldots>\alpha^i.
\end{align*}
Hence $\lim_{t\to\infty}[d_+(0),d_+(t),d_-(t),d_-(0)]=\infty$. This implies $\lim_{t\to\infty}d_+(t)=\lim_{t\to\infty}d_-(t)$ and we have proved $\DD^{\infty}_+=\DD^{\infty}_-=:\DD^{+\infty}$ (and $\DD^{-\infty}_+=\DD^{-\infty}_-=:\DD^{-\infty}$).


We now define $\EE_+=\DD^{-\infty}$ and $\EE_-=\DD^{\infty}$. This choice of signs is the correct one in view of~\eqref{eq:hyper} and the standard definition of the commutator used to orient $\EE=\EE_+\oplus\EE_-$; c.f.~Figure~\ref{b:lines}.

In view of Lemma~\ref{l:Epm are bundles} the continuity of $\EE_\pm$ is automatic, however there is a simple argument in the present situation. 
Let $p\in M$ be arbitrary. If $|T|$ is large enough, then $d_\pm(T)$ are very close to each other at $p$ and for $T>0$ respectively $T<0$, the 
section of $\PP(\EE/\WW)$ which corresponds to $\DD^{\infty}$ respectively $\DD^{-\infty}$ is confined between $d_+(T)$ and $d_-(T)$ near $p$. 
Therefore $\DD^{+\infty}$ and $\DD^{-\infty}$ are continuous plane fields. 

It follows immediately from the definition of $\DD^{\pm\infty}$ that these plane fields are preserved by the holonomy of $\WW$. 
From the condition that $\DD_+$ and $\DD_-$ are always transverse to each other in $\EE$ it follows that $\DD^{+\infty}\neq\DD^{-\infty}$. 

It remains to find a continuous Riemannian metric on $\EE/\WW$ and constants $c>0$ and $K>0$  such that 
\begin{equation} \label{e:partial hyperbolic proof}
\frac{\|D\varphi_t(v_+)\|}{\|v_+\|} \geq Ke^{ct} \frac{\|D\varphi_t(v_-)\|}{\|v_-\|}
\end{equation}
for all $t>0$ and $0\neq v_\pm\in\EE_\pm$. 

Let $X_\pm$ be nowhere vanishing sections of $\EE_\pm$ such that 
\begin{itemize}
\item[(i)] $V=X_+ +X_-$ is smooth and tangent to $\DD_+/\WW$,
\item[(ii)] $X_+$, $X_-$ is a positively oriented framing of $\EE/\WW$ with respect to the orientation defined by $\DD_+$,
\item[(iii)]  $X_+,X_-$ are smooth along the leaves of $\WW$. As above, this can be achieved by convoluting 
$V,X_+,X_-$ with the same bump function. 
\end{itemize}
Because the flow of $W$ preserves $\EE_\pm$ there are continuous functions $\alpha_\pm$ such that 
\begin{equation*}
\left.\frac{d}{dt}\right|_{t=0}\left(D\varphi_{-t}(X_\pm)\right)=\alpha_\pm X_\pm.
\end{equation*}

This implies
\begin{align*}
[W,V] & = \left.\frac{d}{dt}\right|_{t=0}\left(D\varphi_{-t}(X_+ +X_-)\right) \\  
         & = \alpha_+ X_+ + \alpha_-X_-.
\end{align*}
Since $\EE$ is oriented by $W,V,[W,V]$ and this orientation is equivalent to the one given by $W,X_+,X_-$, it follows that 
$\alpha_->\alpha_+$. For all $T\in\R$ there are continuous functions $\lambda_\pm(T)$ on $M$ such that $D\varphi_TX_\pm=\lambda_\pm(T)X_\pm$. These functions satisfy
\begin{align*}
\lambda'_\pm(T)X_\pm &  = \left.\frac{d}{dt}\right|_{t=T}D\varphi_t\left(X_\pm\right) = -D\varphi_T\left(\left.\frac{d}{dt}\right|_{t=0}D\varphi_{-t}X_\pm\right) \\
                                  & = -D\varphi_T(\alpha_\pm X_\pm) = -\left(\alpha_\pm\circ\varphi_{-T}\right) \lambda_\pm(T)X_\pm.
\end{align*}
By definition $\lambda_\pm(T)$ is positive for all $T$. Because of the compactness of $M$, there is a positive number $c$ such that $\alpha_- - \alpha_+>c$. 
Thus we have the following differential inequality
\begin{align*}
\left.\frac{d}{dt}\right|_{t=T} \log\left(\frac{\lambda_+(t)}{\lambda_-(t)}\right) & = \frac{\lambda'_+(T)}{\lambda_+(T)}-\frac{\lambda'_-(T)}{\lambda_-(T)} \\
& = -\alpha_+\circ\varphi_{-T}+\alpha_-\circ\varphi_{-T} \\
& > c.
\end{align*}
If we choose a metric on $\EE/\WW$ for which $X_+$, $X_-$ is an orthonormal frame, then we get the desired inequality~\eqref{e:partial hyperbolic proof} by integration.
\end{proof}

\section{Examples and further discussion}\label{s:final}

\subsection{The Thurston geometry $\Sol$}

The Lie group $\Sol$ is a semidirect product 
$$
1 \lra \Nil^3 \lra \Sol \lra \R \lra 1 \ ,
$$
where $\Nil^3$ is the three-dimensional Heisenberg group, and $\R$ acts by $t\cdot (x,y,z) = (e^{-t}x,e^t y,z)$. The Lie 
algebra of $\Nil^3$ has a basis $X$, $Y$ and $Z$ with $Z$ central and $[X,Y]=Z$. Therefore $X$ and $Y$ span a contact
structure $\xi$ on $\Nil^3$. The action of $\R$ preserves $\xi$ and acts on it contracting $X$ and expanding $Y$. The Lie 
algebra of $\Sol$ has an additional generator $W$ with 
$$
[W,X]=-X \ , \ [W,Y]=Y \ , \ [W,Z]=0 \ .
$$
This means that $X$, $Y$ and $W$ span an even contact structure $\EE$ with $W$ tangent to the characteristic foliation 
$\WW$ of $\EE$. The quotient $\EE/\WW$ is spanned by the images of $X$ and $Y$, and the flow of $W$ is 
hyperbolic on this quotient. Therefore, by Theorem~\ref{t:main}, the distributions $\DD_{\pm}$ spanned by $W$ and 
$X\pm Y$ form a bi-Engel structure. Of course our theorem is not needed in this case, as one can check explicitly that the $\DD_{\pm}$
are Engel structures subordinate to $\EE$ whose intersection is obviously $\WW$, and which induce opposite orientations
on $\EE$. This was done in~\cite{thurston}.

All these structures on $\Sol$ are left-invariant, and therefore descend to closed four-manifolds obtained as quotients by 
lattices. Examples of such quotients are certain mapping tori of $\Nil^3$-manifolds, with the monodromy preserving the 
contact structure induced by $\xi$ on the fibers of the mapping torus.

\subsection{Suspensions of contact-Anosov flows}\label{ss:contactAnosov}

We now want to discuss a large class of bi-Engel structures obtained by suspending contact-Anosov flows. As in the previous example, the 
manifolds we obtain in this way are mapping tori, but the fibers will be very different. 

We  begin with a more general setup in arbitrary dimensions.
Suppose that $\EE$ is an even contact structure with volume-preserving characteristic foliation $\WW$, cf.~Lemma~\ref{l:volume}. We choose a defining form $\alpha$ with $d\alpha$ of constant rank $2n-2$. Any flow tangent to $\WW$ preserves the form $\alpha$, and therefore preserves the symplectic structure\footnote{Here the symplectic structure itself is invariant, not just its conformal class.} defined by $d\alpha$ on $\EE/\WW$. Now assume that the flow of a spanning vector field $W$ of $\WW$ is not just weakly hyperbolic in the sense of Definition~\ref{d:hyper}, but satisfies the following genuine hyperbolicity condition: there exist a continuous metric and a positive constant $b$, such that for the flow $\varphi_t$ of $W$ we have
$$
\vert\vert D\varphi_{t}(v_-)\vert\vert\leq K^{-1}e^{-bt}\vert\vert 
v_-\vert\vert \ \ \ \ \ \forall v_-\in \EE_-,
$$
$$
\vert\vert D\varphi_{t}(v_+)\vert\vert\geq Ke^{bt}\vert\vert 
v_+\vert\vert \ \ \ \ \ \forall v_+\in \EE_+,
$$
for all $t>0$.

\begin{lemma}
In this situation $\EE_{\pm}$ are both of dimension $n-1$, and are Lagrangian for the symplectic structure defined by $d\alpha$ on $\EE_+\oplus\EE_-$.
\end{lemma}

\begin{proof}
Suppose $v,w\in\EE_-$. Then, using $L_{W}\alpha=0$, we find
$$
d\alpha(v,w)=(\varphi_{t}^{*}d\alpha)(v,w)=
d\alpha(D\varphi_{t}(v),D\varphi_{t}(w)).
$$
Using the auxiliary metric $g$, we find that there is a constant $c$ 
such that
$$
\vert d\alpha(v,w)\vert\leq 
c\cdot\vert\vert d\alpha\vert\vert\cdot\vert\vert 
D\varphi_{t}(v)\vert\vert\cdot\vert\vert 
D\varphi_{t}(w)\vert\vert\leq c\cdot\vert\vert d\alpha\vert\vert\cdot
K^{-2}e^{-2bt}\cdot\vert\vert v\vert\vert\cdot\vert\vert w\vert\vert.
$$
Letting $t$ go to infinity, the right-hand-side becomes arbitrarily 
small. Therefore $d\alpha(v,w)=0$, and $\EE_-$ is isotropic for 
$d\alpha$. 
By the analogous argument, letting $t$ go to $-\infty$, we conclude that 
$\EE_+$ is also isotropic. As the two distributions are 
complementary, they must be equidimensional and Lagrangian.
\end{proof}

\begin{example}
Let $N$ be a manifold of dimension $2n-1$, with a contact Anosov vector field $X$. This means that we have a continuous invariant Anosov splitting $TN=\R X\oplus\EE^s\oplus\EE^u$ with the flow $\psi_t$ of $X$ being exponentially contracting on $\EE^s$ and exponentially expanding on $\EE^u$, and that the one-form $\alpha$ with kernel $\EE^s\oplus\EE^u$ and $\alpha (X)=1$ is a contact form. Then $\alpha$ is invariant under $\psi_t$, so that $\alpha$ descends to the mapping torus $M$ of the time one map $\psi_1$. The kernel of $\alpha$ on $M$ is an even contact structure $\EE$. Its characteristic foliation $\WW$ is spanned by the monodromy vector field $W$ of the fibration $M\longrightarrow S^1$. This integrates to a flow $\varphi_t$ on $M$, such that $\varphi_1$ restricted to a fiber coincides with $\psi_1$. Thus the characteristic foliation $\WW$ satisfies the strengthening of the weak hyperbolicity condition described above. 
\end{example}

As the monodromy $\psi_1$ is isotopic to the  identity, the mapping tori $M$ in the example are diffeomorphic to $N\times S^1$. For any $N$ supporting a contact Anosov flow, we obtain an even contact structure on $N\times S^1$ whose characteristic foliation is weakly hyperbolic (and much more).
By the work of Foulon and Hasselblatt~\cite{FH} it is now known that there are very many closed three-manifolds $N$ admitting contact 
Anosov flows. For any such $N$ the product $N\times S^1$ has bi-Engel structures obtained by suspension. Note that by varying the 
time $t$ for which one suspends, one obtains even contact structures on $N\times S^1$ with varying dynamics, e.g. closed orbits of $\WW$
appear and disappear with varying $t$.

\subsection{Rigidity of curves tangent to $\WW$}\label{ss:rigid}

That the Engel planes of a bi-Engel structure never make full turns around $\WW$ leads to a global rigidity property for their 
integral curves tangent to $\WW$. Consider two points $p$ and $q$ in $M$, and let $\Omega_{\DD}(p,q)$ be the space of 
piecewise $C^1$ paths from $p$ to $q$, which are tangent to an Engel structure $\DD$, equipped with the $C^1$ topology. 
As $\DD$ is bracket-generating, the Chow--Rashevskii theorem implies that $\Omega_{\DD}(p,q)$ is non-empty for any pair 
of points. A path in $\Omega_{\DD}(p,q)$  is called rigid, if it has a neighbourhood in $\Omega_{\DD}(p,q)$ such that every 
element of this neighbourhood is a reparametrisation of  the original path. Bryant and Hsu~\cite{BH} proved that a path 
tangent to an Engel structure is rigid if and  only if it is tangent to the characteristic foliation $\WW$, and has the property that 
along the path the Engel plane does not make (more than) a full turn around $\WW$. As a corollary we have:

\begin{proposition}
If an Engel structure $\DD$ is part of a bi-Engel structure, then any path tangent to the characteristic foliation $\WW\subset\DD$ of the induced even contact structure is rigid.
\end{proposition}

The absence of full turns of the bi-Engel planes around $\WW$ is in marked contrast with the properties of the Engel structures constructed by
Casals, P\'erez, del Pino and Presas~\cite{CP3}. Their construction crucially relies on the presence of several full turns along certain
orbits, and therefore never produces this kind of structure. The original existence proof of the second author~\cite{annals} can always
be made to have some leaves of $\WW$ with full turns, but, unless one adds these by hand, it may also produce Engel structures 
without full turns.

There are very few explicit examples of Engel structures known not to have full turns which do not come from bi-Engel structures. 
In~\cite{thurston} Engel structures without full turns were found not only on the Thurston geometry $\Sol$, which is bi-Engel, but 
also on some other solvable geometries and on $\Nil^4$, which are not bi-Engel.

\bigskip

\bibliographystyle{amsplain}

\begin{thebibliography}{999}

\bibitem{AnS}
D.~V.~Anosov and Ya.~G.~Sinai, 
{\em Some smooth ergodic systems}, Russian Math. Surveys~{\bf 22} no.~5 (1967), 103--167.
 
\bibitem{ADN}
M.~Asaoka, E.~Dufraine and T.~Noda,
{\em Homotopy classes of total foliations}, Comment.~Math.~Helv.~{\bf 87} (2012), 271--302.

\bibitem{Be}
M.~Berger, 
{\sl Geometry 1}, Universitext, Springer 1987.

\bibitem{BH}
R.~Bryant and L.~Hsu, {\em Rigidity of integral curves of rank $2$ distributions}, Invent.~Math.~{\bf 114} (1993), 435--461. 

\bibitem{CP3}
R.~Casals, J.~L.~P\'erez, A.~del Pino and F.~Presas, 
{\em Existence $h$-principle for Engel structures}, Invent.~math.~(2017), doi:10.1007/s00222-017-0732-6.

\bibitem{CF}
V.~Colin and S.~Firmo, 
{\em Paires de structures de contact sur les vari\'et\'es de dimension trois}, Algebr.~Geom.~Topol.~{\bf 11} (2011), 2627--2653.

\bibitem{ET}
Y.~M.~Eliashberg and W.~P.~Thurston, 
{\sl Confoliations}, Univ.~Lecture Series vol.~13, Amer.~Math.~Soc., Providence, R.I.~1998.

\bibitem{FH}
P.~Foulon and B.~Hasselblatt, 
{\em Contact Anosov flows on hyperbolic $3$-manifolds}, Geom.~Topol.~{\bf 17} (2013), 1225--1252.

\bibitem{golubev}
A.~Golubev, 
{\em On the global stability of maximally nonholonomic two-plane fields in four dimensions}, 
Int.~Math.~Res.~Not.~1997, no. 11, 523--529.

\bibitem{McD}
D.~McDuff, 
{\em Applications of convex integration to symplectic and contact geometry}, Ann.~Inst.~Fourier {\bf 37} (1987), 107--133.

\bibitem{M}
Y.~Mitsumatsu, 
{\em Anosov flows and non-Stein symplectic manifolds}, Ann.~Inst.~Fourier {\bf 45} (1995), 1407--1421.

\bibitem{M2}
Y.~Mitsumatsu, 
{\em Foliations and contact structures on $3$-manifolds}, in {\sl Foliations: geometry and dynamics}, World Scientific 2002, 75--125.

\bibitem{montgomery}
R.~Montgomery, 
{\em Engel deformations and contact structures}, 
North.~Calif.~Sympl.~Geom.~Sem.~103--117, Amer.~Math.~Soc.~Transl.~Ser. 2, 196 (1999).

\bibitem{PP}
A.~del Pino and F.~Presas, 
{\em Flexibility for tangent and transverse immersions in Engel manifolds}, arXiv:1609.09306v1 [math.SG] 29 Sep 2016.

\bibitem{thurston}
T.~Vogel, 
{\em Maximally non-integrable plane fields on Thurston geometries},  Int.~Math.~Res.~Not.~2006, Art.~ID 97315, 29 pp. 

\bibitem{annals}
T.~Vogel, 
{\em Existence of Engel structures}, Ann.~of~Math.~{\bf 169} (2009), 79--137.

\end{thebibliography}

\bigskip

\end{document}